\documentclass[11pt]{amsart}
\usepackage{amsmath}
\usepackage{amssymb}
\usepackage{amscd}
\usepackage{enumerate}
\usepackage{verbatim}
\usepackage{color}
\usepackage{bm}
\usepackage{easybmat} 
\usepackage{graphicx}
\usepackage{subfigure}

\newtheorem{Thm}{Theorem}[section]
\newtheorem{Lemma}[Thm]{Lemma}

\theoremstyle{definition}

\newtheorem{Example}[Thm]{Example}

\newtheorem{Rmk}[Thm]{Remark}

\def\bfa{\mathbf{a}}

\def\bfe{\mathbf{e}}

\def\bfp{\mathbf{p}}

\def\bfx{\mathbf{x}}

\def\frakG{\mathfrak{G}}

\def\bba{\mathbb{A}}

\def\bbd{\mathbb{D}}

\def\bbn{\mathbb{N}}

\def\bbr{\mathbb{R}}

\def\bbz{\mathbb{Z}}

\def\calS{\mathcal{S}}

\def\lra{\longrightarrow}

\def\x{\times}

\def\aff{\mathrm{Aff}}

\def\Nil{\mathrm{Nil}}

\def\ab{\mathrm{ab}}

\def\sp{\mathrm{sp}}
\def\diag{\mathrm{diag}}

\def\Sol{\mathrm{Sol}}
\def\GammaA{\Gamma_{\!A}}
\def\bfa{\mathbf{a}}

\def\R{\mathrm{(R)}}

\def\bb{|\hspace{-1pt}|}

\def\ds{\displaystyle}

\def\GR{\mathrm{GR}}
\def\top{\mathrm{top}}
\def\alg{\mathrm{alg}}

\def\curlle{\preccurlyeq}

\def\boxit#1{\vbox{\hrule\hbox{\vrule\kern3pt
     \vbox{\kern3pt#1\kern3pt}\kern3pt\vrule}\hrule}}

\begin{document}
\title[Growth rate and finitely generated nilpotent groups]
{Growth rate for endomorphisms of finitely generated nilpotent groups and solvable groups}
\author{Alexander Fel'shtyn}
\address{Instytut Matematyki, Uniwersytet Szczecinski,
ul. Wielkopolska 15, 70-451 Szczecin, Poland}
\email{fels@wmf.univ.szczecin.pl}

\author{Jang Hyun Jo}
\address{Department of mathematics, Sogang University, Seoul 121-742, KOREA}
\email{jhjo@sogang.ac.kr}

\author{Jong Bum Lee}
\address{Department of mathematics, Sogang University, Seoul 121-742, KOREA}
\email{jlee@sogang.ac.kr}


\abstract
We prove that the growth rate of an
endomorphism of a finitely generated nilpotent group  equals to the
growth rate of induced endomorphism on its abelinization, generalizing the corresponding result
for an automorphism in \cite{Koberda}.
We also study growth rates of endomorphisms
for specific solvable groups, lattices of $\Sol$, providing a counterexample to a
known result in \cite{FFK} {and proving that the growth rate is an algebraic number}.
\endabstract
\maketitle

\tableofcontents

\section{Introduction}\label{introd}

In the present paper we study  purely algebraic notions of growth rate
and entropy for an endomorphism of a finitely generated group.

Let $\pi$ be a finitely generated group with a system $S=\{s_1,\cdots,s_n\}$ of generators.
Let $\phi:\pi\to\pi$ be an endomorphism.
For any $\gamma\in\pi$,
let $L(\gamma,S)$ be the length of the shortest word in the letters $S\cup S^{-1}$
which represents $\gamma$.
Then the {\bf growth rate} of $\phi$ is defined to be (\cite{B78})
$$
\GR(\phi):=\sup\left\{\limsup_{k\to\infty}L(\phi^k(\gamma),S)^{1/k}\mid \gamma\in\pi\right\}.
$$
For each $k>0$, we put
$$
L_k(\phi,S)=\max \left\{L(\phi^k(s_i),S)\mid i=1,\cdots,n\right\}.
$$
It is known that
$$
\GR(\phi)=\lim_{k\to\infty}L_k(\phi,S)^{1/k}
=\inf_{k}\left\{L_k(\phi,S)^{1/k}\right\},
$$
and the {\bf algebraic entropy} of $\phi$ is by definition $h_\alg(\phi):=\log\GR(\phi)$.
The growth rate and hence the algebraic entropy of $\phi$ are well-defined,
i.e., independent of the choice of a set of generators  (\cite[p.~\!114]{KH}).
It is immediate from the definition that the growth rate
and the algebraic entropy for an endomorphism of a group
are invariants of conjugacy of group endomorphisms.
Furthermore, for any inner automorphism $\tau_{\gamma_0}$ by $\gamma_0$,
we have $\GR(\tau_{\gamma_0}\phi) = \GR(\phi)$ and $h_\alg(\tau_{\gamma_0}\phi)=h_\alg(\phi)$
(\cite[Proposition~3.1.10]{KH}).

Consider a continuous map $f$ on a compact connected manifold $M$, and
consider a homomorphism $\phi$ induced by $f$ of
the group of covering transformations on the universal cover of $M$.
Then the topological entropy $h_\top(f)$ is defined.
We refer to \cite{KH} for backgrounds.
Among others, we recall that R. Bowen in \cite{B78} and A. Katok in \cite{K} have proved that
the topological entropy $h_\top(f)$ of $f$ is at least as large as the algebraic entropy
$h_\alg(\phi)=h_\alg(f)$ of $\phi$ or $f$.

The problem of determining the growth rate of a group endomorphism,
initiated  by R. Bowen in \cite{B78}, is now an area of active research
(see detailed description in \cite{FFK} and \cite{Koberda} and references therein).
{For known properties of the growth of automorphisms of free groups we refer to \cite{BFH, LL}.}

The purpose of this paper is first to study the growth rate of an endomorphism
on a finitely generated nilpotent group.
In \cite[Theorem~1.2]{Koberda} it was proven that the growth rate of an
automorphism of a finitely generated nilpotent group is equal to the
growth rate of induced automorphism on its abelinization.
Our main result  is a generalization of this result  of  \cite{Koberda}
from automorphisms to endomorphisms (Theorem~\ref{Ko type}) by using
completely different arguments.
{In Section~\ref{prelim} we remind some known results about growth rate of group  endomorphism, sometimes correcting them.}
In Section~\ref{nilpotent} we refine the calculation  in \cite{B78} of the growth rate
for an endomorphism of a finitely generated torsion-free nilpotent group
and prove that the growth rate  is an algebraic number.
Another purpose of this article is to begin a study of growth rates
for specific solvable groups, lattices of $\Sol$, providing a counterexample to Theorem~5.1 in \cite{FFK}
{and proving that the growth rate is an algebraic number.}
\medskip

\noindent
\textbf{Acknowledgments.}
The first author is indebted to the Max-Planck-Institute for Mathematics(Bonn)
and Sogang University(Seoul) for the support and hospitality
and the possibility of the present research during his visits there.
The second author was supported by Basic Science Research
Program through the National Research Foundation of Korea(NRF)
funded by the Ministry of Education, Science and Technology (2012R1A1A2006395).
The third author was supported by Basic Science Researcher Program
through the National Research Foundation of Korea(NRF) funded by the Ministry of Education
(No.~\!2013R1A1A2058693) and by the Sogang University Research Grant of 2010 (10022)

\section{Preliminaries}\label{prelim}

We shall assume in this article that {\bf all groups are finitely generated}
unless otherwise specified. For a given endomorphism $\phi:\pi\to\pi$,
if $\pi'$ is a $\phi$-invariant subgroup of $\pi$,
we denote by $\phi'=\phi|_{\pi'}$ the restriction of $\phi$ on $\pi'$.
If, in addition, $\pi'$ is a normal subgroup,
we denote by $\hat\phi$ the endomorphism on $\pi/\pi'$ induced by $\phi$.
Then the following are known, see for example \cite{B78,FFK}:
\begin{itemize}
\item $\GR(\phi^k)=\GR(\phi)^k$ for $k>0$.
\item $\GR(\hat\phi)\le\GR(\phi)$.
\item {$\GR(\phi)\le\max\left\{\GR(\phi'),\GR(\hat\phi)\right\}$.}
\item Let $\phi:\bbz^n\to\bbz^n$ be an endomorphism yielding an integer matrix $D$.
Then we have $\GR(\phi)=\sp(D)$, the maximum of absolute values of eigenvalues of $D$.
\end{itemize}

Let $S'$ be a finite set of generators for $\pi'$
and let $\hat{S}$ be a finite set of generators for the quotient group $\pi/\pi'$.
Then it is possible to extend $S'$ to a finite set $S$ of generators for $\pi$
so that $S$ is projected onto $\hat{S}$ under the projection $\pi\to\pi/\pi'$.
For any $\gamma\in\pi'$, it is true that $L(\gamma,S')\ge L(\gamma,S)$.

Consider the concentric balls $B(n)=\left\{\gamma\in\pi\mid L(\gamma,S)\le n\right\}$ for all $n>0$,
and the {\bf distortion function} of $\pi'$ in $\pi$ which is defined as
$$
\Delta^\pi_{\pi'}(n):={\max}\left\{L(\gamma,S')\mid \gamma\in\pi'\cap B(n)\right\}.
$$
The notion of distortion of a subgroup was first introduced by M.~Gromov in \cite{Gromov}.
We refer to \cite{Davis} for our discussion. For two functions $f,g:\bbn\to\bbn$, we say that $f\curlle g$
if there exists $c>0$ such that such that $f(n)\le cg(cn)$ for all $n>0$.
We say that two functions are equivalent,
written $f\approx g$, if $f\curlle g$ and $g\curlle f$.
The subgroup $\pi'$ of $\pi$ is {\bf undistorted}
if $\Delta^\pi_{\pi'}(n)\approx n$.
The following facts about distortion can be found in \cite{Davis}:
\begin{itemize}
\item If $\pi'$ is infinite, then it is true that $n\curlle\Delta^\pi_{\pi'}(n)$.
\item If $[\pi:\pi']<\infty$, then $\pi'$ is undistorted in $\pi$.
\end{itemize}

Assume that $\Delta^\pi_{\pi'}(n)\curlle n$. By definition, there exists $c>0$
such that $\Delta^\pi_{\pi'}(n)\le c^2n$ for all $n>0$.
For any $\gamma\in\pi'$, let $n=L(\gamma,S)$. Then
$$
L(\gamma,S')\le\Delta^\pi_{\pi'}(n)\le c^2n=c^2 L(\gamma,S).
$$
Thus $L(\gamma,S)\le c^2L(\gamma,S)$ for all $\gamma\in\pi'$.
This inequality induces that for all $k>0$
\begin{align*}
L_k(\phi',S')&=\max\left\{L({\phi'}^k(\gamma_i),S')\mid \gamma_i\in S'\right\}\\
&\le c^2\max\left\{L({\phi'}^k(\gamma_i),S)\mid \gamma_i\in S'\right\}
\le c^2L_k(\phi,S)
\end{align*}
and so $\GR(\phi')\le\GR(\phi)$. Consequently, we have

\begin{Lemma}[{\cite[Corollary~3.1]{FFK}}]
Let $\phi$ be an endomorphism of $\pi$. If $\pi'$ is a $\phi$-invariant undistorted subgroup in $\pi$,
then $\GR(\phi')\le\GR(\phi)$; {hence if, in addition, $\pi'$ is a normal subgroup of $\pi$, then $\GR(\phi)=\max\left\{\GR(\phi'),\GR(\hat\phi)\right\}$.}
\end{Lemma}

\begin{proof}
Since $\pi'$ is undistorted in $\pi$, we have from definition that $\Delta^\pi_{\pi'}(n)\curlle n$.
Now the proof follows from the above observation.
\end{proof}

\begin{Rmk}
Remark further that:
\begin{itemize}
\item {If $\pi'$ is of finite index in $\pi$, then $\pi'$ is undistorted
    and hence $\GR(\phi')\le\GR(\phi)$. Example~\ref{counter} shows that the inequality can be strict.
    Thus \cite[Proposition~1]{B78} (see also \cite[Theorem~3.1]{FFK}) is not correct.}
\item If $\GR(\phi)<\GR(\phi')$, then $\pi'$ is distorted, and $\pi'$ is not of finite index in $\pi$.
\end{itemize}
\end{Rmk}

\begin{Lemma}\label{GR=0}
Let $\phi$ be an endomorphism of $\pi$. If $\GR(\phi)<1$,
then $\GR(\phi)=0$ and $\phi$ is an eventually trivial endomorphism, {and vice versa.}
\end{Lemma}

\begin{proof}
Let ${\rho}=\GR(\phi)$ and let $\epsilon=1-{\rho}>0$.
Since $\lim_{m\to\infty}L_m({\phi},S)^{1/m}={\rho}$,
there exists $N>0$ such that for all $m\ge N$ we have $L_m({\phi},S)^{1/m}-{\rho}<\epsilon$;
$L_m(\phi,S)^{1/m}<1\Rightarrow L_m(\phi,S)<1\Rightarrow L_m(\phi,S)=0$
because $L_m(\phi,S)$ is a nonnegative integer.
This implies that ${\rho}=0$ and the endomorphism ${\phi}^N$ is trivial or $\phi$ is eventually trivial.
{The converse is obvious.}
\end{proof}

\begin{Example}\label{counter}
Let $\pi=\bbz\x\bbz_2$ with generators $\alpha$ and $\beta$ such that $\beta^2=1$.
Consider an endomorphism $\phi$ of $\pi$ defined by $\phi(\alpha)=1$ and $\phi(\beta)=\beta$.
Observing that
\begin{align*}
L_n(\phi,S)&=\max\left\{L(\phi^n(\alpha),S), L(\phi^n(\beta),S)\right\}\\
&=\max\left\{L(1,S),L(\beta,S)\right\}
=\max\left\{0,1\right\}=1,
\end{align*}
we have $\GR(\phi)=1$. Similarly, we have $\GR(\phi|_\bbz)=0$ and $\GR(\phi|_{\bbz_2})=1$.
Notice further that $\bbz_2$ is a distorted subgroup of $\pi$ because $\Delta^\pi_{\bbz_2}(n)=1$ for all $n$.
\end{Example}

\begin{Lemma}\label{finite}
Let $\phi$ be an endomorphism of $\pi$.
\begin{enumerate}
\item[$(1)$] If $\pi'$ is a $\phi$-invariant finite subgroup of $\pi$,
then $\GR(\phi')\le\GR(\phi)$;
\item[$(2)$] {If, in addition, $\pi'$ is a normal subgroup of $\pi$, then $\GR(\phi)\!=\!\max\left\{\GR(\phi'),\GR(\hat\phi)\right\}$, and
    $\GR(\phi)=\GR(\hat\phi)$ if and only if $\phi'$ is eventually trivial or $\hat\phi$
    is not eventually trivial.}
\end{enumerate}
\end{Lemma}

\begin{proof}
If the $\phi$-invariant subgroup $\pi'$ of $\pi$ is finite,
then we can show easily that $\GR(\phi')$ is either $0$ or $1$
by taking a system of generators $S'=\pi'$ for $\pi'$.
We will show that $\GR(\phi')\le\GR(\phi)$. May assume that $\GR(\phi')=1$.
This implies that there is an element $x\in\pi'$ such that ${\phi'}^n(x)\ne1$ for all $n>0$.
Considering any system of generators for $\pi$ which contains $x$,
we can see right away that $\GR(\phi)\ge1=\GR(\phi')$.

{Assume that $\pi'$ is normal in $\pi$. If $\GR(\phi')=0$, then it is clear that
$\GR(\phi)=\GR(\hat\phi)$. On the other hand, if $\GR(\phi')=1$ then $\GR(\phi)=\GR(\hat\phi)$
if and only if $\GR(\hat\phi)\ge1$ if and only if $\hat\phi$ is not eventually trivial by Lemma~\ref{GR=0}.}
\end{proof}

\begin{Rmk}
However the above lemma is not true anymore whenever $\pi'$ is infinite, see also Example~\ref{BS}.
Note further that if $\GR(\phi)<\GR(\phi')$, then $\pi'$ is infinite.
\end{Rmk}

The following is an well-known example about subgroup distortion.

\begin{Example}\label{BS}
Let $\pi$ be the Baumslag-Solitar group $B(1,n)$:
$$
B(1,n):= \left\langle a,b\mid a^{-1}ba = b^n \right\rangle,\ n>1.
$$
Then $S = \{a,b\}$ is a generating set for $\pi$. Let $\pi'=\langle b\rangle$ and let $S'=\{b\}$.
We observe that the subgroup $\pi'$ of $\pi$ is distorted.
In fact, since $b^{n^{k}} = a^{-k}ba^{k}$ for all $k>0$,
we have that $L(b^{n^k},S')=n^k$ and $L(b^{n^k},S)=2k+1$.
{If $\phi$ is an endomorphism of $\pi$ given by $\phi(b)=b^n$ and $\phi(a)=a$, then
we can see that $\GR(\phi')=n$ and $\GR(\phi)=1$.}
\end{Example}

{Example~\ref{counter} shows that \cite[Proposition~1]{B78} is not correct in general,
but it is almost true in the sense of Theorem~\ref{f-index}.}
By modifying the argument of the proof of \cite[Theorem~3.1]{FFK}, we have:

\begin{Thm}\label{f-index}
Let $\phi$ be an endomorphism of $\pi$ and let $\pi'$ be a $\phi$-invariant, finite index subgroup of $\pi$.
\begin{enumerate}
\item[$(1)$] If $\phi'$ is not an eventually trivial endomorphism, then $\GR(\phi)=\GR(\phi')$.
\item[$(2)$] If $\phi'$ is an eventually trivial endomorphism of $\pi'$,
then $\GR(\phi')=0$ and $\GR(\phi)=0$ or $1$.
Moreover, $\GR(\phi)=0$ if and only if $\phi$ is an eventually trivial endomorphism of $\pi$.
\end{enumerate}
Consequently, the equality
$\GR(\phi)=\GR(\phi')$ holds except only the case when $\phi'$ is eventually trivial and $\phi$ is not eventually trivial. If this is the case, then $\GR(\phi')=0$ and $\GR(\phi)=1$.
\end{Thm}

\begin{proof}
%
Let $S'=\{\gamma_1,\cdots,\gamma_t\}$ be a set of generators of
$\pi'$. Let $u=[\pi:\pi']$. Then we have $\pi=\delta_1\pi'\cup\cdots\cup\delta_u\pi'$
so that $S=\{\gamma_1,\cdots,\gamma_t,\delta_1,\cdots,\delta_u\}$ generates $\pi$.
For any $j=1,\cdots,u$, there exists a unique $k_j$ such that $\phi(\delta_j)\in \delta_{k_j}\pi'$.
We denote
$$
p=\max_{1\le j\le u}\left\{L(w_j,{S'})\mid \phi(\delta_j)=\delta_{k_j}w_j\in\delta_{k_j}\pi'\right\}.
$$

Assume that $p=0$. Then $\phi(\delta_j)=\delta_{k_j}$ for all $j=1,\cdots,u$.
For each $j=1,\cdots,u$, we write $\phi^m(\delta_j)=\delta_{j_m}$.
Hence $L(\phi^m(\delta_j),S)=0$ or $1$ according as $\delta_{j_m}=1$ or $\delta_{j_m}\ne1$.

Suppose that there is $N>0$ such that $\phi^N(\delta_j)=1$ for all $j=1,\cdots,u$
and hence $L(\phi^m(\delta_j),S)=0$ for all $m\ge N$.
Since $\pi'$ is undistorted in $\pi$, there exists $c>0$ such that
\begin{align*}
L(\gamma,S')\le c^2\cdot L(\gamma,{S}),\quad \forall\gamma\in\pi'.
\end{align*}
It is clear that
\begin{align*}
L(\gamma,S)\le L(\gamma,{S'}),\quad \forall\gamma\in\pi'.
\end{align*}
Thus,
\begin{align*}
&L_m(\phi',S')\le c^2\cdot L_m(\phi,S),\\
&L_m(\phi,S)
=\max\left\{L(\phi^m(\gamma_i),S)\right\}\le L_m(\phi',S').
\end{align*}
This implies that $\GR(\phi')=\GR(\phi)$.

Suppose on the contrary that for any $m>0$ there is $j$ such that $\phi^m(\delta_j)\ne1$.
Then $\max\left\{L(\phi^m(\delta_j),S)\right\}=1$. Hence
\begin{align*}
L_m(\phi,S)&=\max\left\{L(\phi^m(\gamma_i),S), L(\phi^m(\delta_j),S)\right\}\\
&=\max\left\{L(\phi^m(\gamma_i),S),1\right\}
\le \max\left\{L_m(\phi',S'),1\right\}.
\end{align*}
This implies that $\GR(\phi')\le\GR(\phi)\le\max\left\{\GR(\phi'),1\right\}$.
Since $\phi'$ is not eventually trivial, Lemma~\ref{GR=0} implies that $\GR(\phi')\ge1$, and hence
$\GR(\phi)=\GR(\phi')$.

Next we assume next that $p\ge1$.
For each $j=1,\cdots,u$, we write $\phi(\delta_j)=\delta_{j_1}w_1$ for some $j_1$ and $w_1\in\pi'$.
Then
\begin{align*}
\phi^m(\delta_j)=\delta_{j_m}w_m\phi(w_{m-1})\cdots\phi^{m-1}(w_1).
\end{align*}
and thus
$$
L(\phi^m(\delta_j),S)\le 1+p+pL_1(\phi',S')+pL_2({\phi'},S')+\cdots+pL_{m-1}({\phi'},S').
$$
Let $L=\GR(\phi')$. By the assumption of our proposition, $L\ge1$. Let $\epsilon>0$ be given. Since $\lim_{m\to\infty}L_m({\phi'},S')^{1/m}=L$,
there is $N>0$ such that if $m>N$ then $L_m({\phi'},S')<(L+\epsilon)^m$.
Choose $q_1,\cdots,q_N>0$ such that $L_i({\phi'},S')< q_i(L+\epsilon)^i$ for $i=1,\cdots,N$.
Put $q=\max\left\{q_1,\cdots,q_N,1\right\}\ge1$. Then $L_m({\phi'},S')<q(L+\epsilon)^m$ for all $m\ge1$. Hence we have
\begin{align*}
L(\phi^m(\delta_j),S)&\le 1+p+pq(L+\epsilon)+pq(L+\epsilon)^2+\cdots+pq(L+\epsilon)^{m-1}\\
&\le 1+pq\frac{(L+\epsilon)^m-1}{(L+\epsilon)-1}.
\end{align*}
Since $pq\ne0$, this induces that
$$
\lim_{m\to\infty}\sqrt[m]{\max_{j}\left\{L(\phi^m(\delta_j),S)\right\}}\le L+\epsilon.
$$
Since $\pi'$ is undistorted in $\pi$, there exists $c>0$ such that
\begin{align*}
L_m(\phi',S')&=\max_{i}\left\{L({\phi'}^m(\gamma_i),S')\right\}\\
&\le c^2\cdot\max_{i,j}\left\{L({\phi'}^m(\gamma_i), {S}),L(\phi^m(\delta_j),S)\right\}=c^2L_m(\phi,S),
\end{align*}
and hence we obtain
$$
L=\GR(\phi')\le\GR(\phi)=\lim_{m\to\infty}\sqrt[m]{L_m(\phi,S)}\le L+\epsilon
$$
for all $\epsilon>0$. Consequently, $\GR(\phi)=\GR(\phi')$.

Suppose {that} $\phi'$ is an eventually trivial endomorphism of $\pi'$.
Then {it is clear that} $\GR(\phi')=0$. Consider a set $S=\left\{\gamma_1,\cdots,\gamma_t,\delta_1,\cdots,\delta_u\right\}$ of generators for $\pi$ as above.
For any $m>0$, we observe that
$$
\phi^m(\gamma_i)=1,\quad
\phi^m(\delta_j)=\delta_{j_m}w_m
$$
for some $j_m\in\left\{1,\cdots,u\right\}$ and $w_m$ in a finite subset of $\pi'$.
This induces that the sequence $\left\{L_m(\phi,S)\right\}$ is bounded.
Because $L_m(\phi,S)=0$ or $\ge1$, it follows that $\GR(\phi)=0$ or $1$ respectively.

When $\GR(\phi)=0$, Lemma~\ref{GR=0} says that $\phi$ is an eventually trivial endomorphism.
Next we consider the case when $\GR(\phi)=1$. From the definition, we can choose $N>0$ so that for $m\ge N$
we have $1/2^m<L_m(\phi,S)$, which implies that $L_m(\phi,S)\ge1$ because $L_m(\phi,S)$ is an integer. Therefore, for each $m\ge N$, we can choose $\gamma\in S$ such that $\phi^m(\gamma)\ne1$. This shows that $\phi$ is not eventually trivial even though $\phi'$ is eventually trivial.
\end{proof}

Before leaving this section, we observe the following elementary facts.
These turn out to be useful in driving practical computation formula of $\GR(\theta)$, see Section~\ref{Sol-lattice}.

\begin{Lemma}\label{max-lim}
Let $\phi$ be an endomorphism of $\pi$ with a finite set $S$ of generators.
Let
$$
\GR_i(\phi)=\lim_{k\to\infty} L(\phi^k(s_i),S)^{1/k}
$$
for each $s_i\in S$. Then $\GR(\phi)=\max\left\{\GR_i(\phi)\mid s_i\in S\right\}$.
\end{Lemma}

\begin{proof}
Since $L(\phi^k(s_i),S)\le L_k(\phi,S)$, it follows that $\GR_i(\phi)\le\GR(\phi)$.
Assume $\GR_i(\phi)<\GR(\phi)$ for all $s_i\in S$. Thus
there exists $K>0$ such that if $k\ge K$ and $s_i\in S$ then $L(\phi^k(s_i),S)^{1/k}<\GR(\phi)$.
Because $S$ is finite, it follows that $L_k(\phi,S)^{1/k}<\GR(\phi)$ for all $k\ge K$.
However, since $\lim_{k\to\infty}L_k(\phi,S)^{1/k}=\lim_{k\ge K}L_k(\phi,S)^{1/k}=\inf_{k\ge K}L_k(\phi,S)^{1/k}$, we obtain a contradiction: $\GR(\phi)=\inf_{k\ge K}L_k(\phi,S)^{1/k}<\GR(\phi)$.
\end{proof}

\begin{Lemma}\label{tech1}
Assume that $f(k), g(k)\ge0$, $\ds{\lim_{k\to\infty} f(k)^{1/k}}=F$ and $\ds{\lim_{k\to\infty} g(k)^{1/k}}=G$.
Then for any positive constants $A$ and $B$, we have
$$
\lim_{k\to\infty}\left(Af(k)+Bg(k)\right)^{1/k}=\lim_{k\to\infty}\left(f(k)+g(k)\right)^{1/k}=\max\left\{F,G\right\}.
$$
\end{Lemma}

\begin{proof}
If $\lim_{k\to\infty} f(k)^{1/k}=F$, then
$\lim_{k\to\infty} (Af(k))^{1/k}=\lim_{k\to\infty} A^{1/k}\cdot\lim_{k\to\infty} f(k)^{1/k}=F$.
This implies that we can assume $A=B=1$.

If $G=0$, then for sufficiently large $k$, $g(k)^{1/k}<\frac{1}{2}$ or $g(k)<\frac{1}{2^k}$,
which implies that $g(k)\to0$. Hence
$\lim_{k\to\infty}(f(k)+g(k))^{1/k}=\lim_{k\to\infty} f(k)^{1/k}=F=\max\left\{F,G\right\}$.

We may now assume that $0<G\le F$. The assumption $\lim_{k\to\infty} f(k)^{1/k}=F$ deduces
that $\lim_{k\to\infty} \frac{f(k)^{1/k}}{F}=1$
and it follows that $\lim_{k\to\infty}\frac{f(k)}{F^k}=1$.
Similarly, $\lim_{k\to\infty}\frac{g(k)}{G^k}=1$. So, $\lim_{k\to\infty} \frac{g(k)}{F^k}=\lim_{k\to\infty} \frac{g(k)}{G^k}\left(\frac{G}{F}\right)^k$
is $1$ or $0$ according as $G=F$ or $G<F$. Therefore
\begin{align*}
\log\left(\lim_{k\to\infty}\frac{(f(k)+g(k))^{1/k}}{F}\right)
&=\lim_{k\to\infty}\frac{\log\frac{f(k)+g(k)}{F^k}}{k}\\
&=\lim_{k\to\infty}\frac{\log\left(\frac{f(k)}{F^k}+\frac{g(k)}{F^k}\right)}{k}=0.
\end{align*}
This proves our assertion.
\end{proof}

\section{Finitely generated nilpotent groups}\label{nilpotent}

Consider the lower central series of a finitely generated group
$\pi$: $\pi=\pi_1\supset\pi_2\supset\cdots$,
where $\pi_{j}=[\pi,\pi_{j-1}]$ is the $j$-fold commutator subgroup $\gamma_j(\pi)$ of $\pi$.
The endomorphism $\phi:\pi\to\pi$ induces endomorphisms
$$
\phi_j:\pi_j\to\pi_j,\
\hat\phi_j:\pi/\pi_j\to\pi/\pi_j,\
\bar\phi_j:\pi_j/\pi_{j+1}\to\pi_j/\pi_{j+1}.
$$
Then it is known from \cite{B78} that $\GR(\phi)\ge\GR(\bar\phi_j)^{1/j}$ for all $j\ge1$.
The group $\pi$ is called {\bf nilpotent} if $\pi_j=1$ for some $j$. When $\pi_c\ne1$ but $\pi_{c+1}=1$,
we say that it is {\bf $c$-step}. It is also known from \cite{B78} that:
\begin{itemize}
\item If $\pi$ is $c$-step nilpotent, then $\GR(\phi)=\max\left\{\GR(\hat\phi_c), \GR(\phi_c)^{1/c}\right\}$,
\item If $\pi$ is nilpotent, then $\GR(\phi)= \max_{j\ge1}\left\{\GR(\bar\phi_j)^{1/j}\right\}$.
\end{itemize}

Recall for example from \cite[Proposition~3.1]{Koberda} that a finitely generated nilpotent group $\pi$
is virtually torsion-free. Thus there exists a finite index, torsion-free, normal subgroup $\Gamma$ of $\pi$.
Following the proof of \cite[Lemma~3.1]{LL-JGP}, we can see that there exists a fully invariant subgroup $\Lambda\subset\Gamma$
of $\pi$ which is of finite index. Therefore, any endomorphism $\phi:\pi\to\pi$ restricts to
an endomorphism $\phi':\Lambda\to\Lambda$.
By Theorem~\ref{f-index}, we may consider only the case when $\phi'$ is not eventually trivial
and hence we may assume that $\GR(\phi)=\GR(\phi')$.
Consequently, for the computation of $\GR(\phi)$, we may assume that $\pi$ is a finitely generated torsion-free nilpotent group.

Consider the lower central series of a finitely generated torsion-free $c$-step nilpotent group
$\pi$: $\pi=\pi_1,\ \pi_{j+1}=[\pi,\pi_j]$, $\pi_c\ne1$ and $\pi_{c+1}=1$.
For each $j=1,\cdots,c$, we consider the isolator of $\pi_j$ in $\pi$:
$$
\sqrt{\pi_j}=\sqrt[\pi]{\pi_j}:=\left\{x\in\pi\mid x^k\in\pi_j \text{ for some $k\ge1$}\right\}.
$$
Then it is known that $\sqrt{\pi_j}$ is a characteristic subgroup of $\pi$
with $[\sqrt{\pi_j}:\pi_j]$ is finite.
Furthermore, $\sqrt{\pi_j}/\pi_j$ is precisely the set of all torsion elements
in the nilpotent group $\pi/\pi_j$ and $\sqrt{\pi_j}/\sqrt{\pi_{j+1}}\cong\bbz^{k_j}$ for some integer $k_j>0$.
Hence we obtain the {\bf adapted central series}
$$
\pi=\sqrt{\pi_1}\supset\sqrt{\pi_2}\supset\cdots\supset\sqrt{\pi_c}\supset\sqrt{\pi_{c+1}}=1.
$$

The following lemma plays a crucial role in our study of growth rate for endomorphisms
of finitely generated nilpotent groups.

\begin{Lemma}[{\cite[Lemma~3.7]{Wolf68}}]\label{W-basis}
Let $\pi$ be a finitely generated $c$-step nilpotent group with lower central series
$$
\pi={\pi_1}\supset{\pi_2}\supset\cdots\supset{\pi_c}\supset{\pi_{c+1}}=1.
$$
Then there are finite sets $T_j=\left\{\tau_{j1},\cdots,\tau_{jk_j}\right\}\subset{\pi_j}$ such that
\begin{enumerate}
\item[$(1)$] if $p_j:{\pi_j}\to{\pi_j}/{\pi_{j+1}}$ denotes the projection,
then $p_j(T_j)$ is an independent set of generators for the finitely generated Abelian group ${\pi_j}/{\pi_{j+1}}$;
\item[$(2)$] if $j>1$, then every $\tau_{jr}$ is of the form $[\tau_{1i},\tau_{j-1,\ell}]$; and
\item[$(3)$] $T_1$ generates $\pi$.
\end{enumerate}
\end{Lemma}

Let $G$ be the Malcev completion of a finitely generated torsion-free nilpotent group
and let $\phi$ be an endomorphism of $\pi$.
Then $\phi$ extends uniquely to a Lie group homomorphism $D$ of $G$, called the Malcev completion of $\phi$.
We call its differential $D_*$ the {\bf linearization} of $\phi$.

\begin{Thm}\label{GR-nilp}
Let $\phi:\pi\to\pi$ be an endomorphism on a finitely generated torsion-free nilpotent group $\pi$.
Let $G$ be the Malcev completion of $\pi$. Then the linearization $D_*:\frakG\to \frakG$ of $\phi$ can be expressed as a lower triangular block matrix with diagonal blocks $\left\{D_j\right\}$ so that
$$
\GR(\phi)=\max_{j\ge1}\ \left\{\sp(D_j)^{1/j}\right\}.
$$
In particular, $\GR(\phi)$ is an algebraic integer.
\end{Thm}

\begin{proof}
Let $\pi$ be a finitely generated torsion-free $c$-step nilpotent group with adapted central series
$$
\pi=\sqrt{\pi_1}\supset\sqrt{\pi_2}\supset\cdots\supset\sqrt{\pi_c}\supset\sqrt{\pi_{c+1}}=1.
$$
Let $q_j:\sqrt{\pi_j}\to\sqrt{\pi_j}/\sqrt{\pi_{j+1}}$ denotes the projection.
We choose $\left\{T_1,\cdots,T_c\right\}$ as in Lemma~\ref{W-basis}.
Since $\pi_2$ is a fully invariant, finite index subgroup of $\sqrt{\pi_2}$, it induces a short exact sequence
$$
1\lra\sqrt{\pi_2}/\pi_2\lra\pi_1/\pi_2\lra\pi_1/\sqrt{\pi_2}=\sqrt{\pi_1}/\sqrt{\pi_2}\lra1.
$$
Since $\sqrt{\pi_2}/\pi_2$ is finite, it follows that $\sqrt{\pi_1}/\sqrt{\pi_2}\cong\bbz^{k_1}$
can be regarded as the free part of the finitely generated Abelian group $\pi_1/\pi_2$.
Hence we can choose $S_1\subset T_1$ such that $p_1(S_1)$ is an independent set of free generators of $\sqrt{\pi_1}/\sqrt{\pi_2}$ and $p_1(T_1-S_1)$ is an independent set of torsion generators of ${\pi_1}/{\pi_2}$.

Next we consider the short exact sequence
$$
1\lra\sqrt{\pi_3}/\pi_3\lra\sqrt{\pi_2}/\pi_3\lra\sqrt{\pi_2}/\sqrt{\pi_3}\lra1.
$$
Since $\pi_2/\pi_3\subset\sqrt{\pi_2}/\pi_3$, we obtain the following commutative diagram between exact sequences
$$
\CD
1@>>>\sqrt{\pi_3}/\pi_3@>>>\sqrt{\pi_2}/\pi_3@>>>\sqrt{\pi_2}/\sqrt{\pi_3}\cong\bbz^{k_2}@>>>1\\
@.@AAA@AAA@AAA\\
1@>>>(\pi_2\cap\sqrt{\pi_3})/\pi_3@>>>{\pi_2}/\pi_3@>>>\begin{matrix}{\pi_2}/(\pi_2\cap\sqrt{\pi_3})\\
\quad=\pi_2\cdot\sqrt{\pi_3}/\sqrt{\pi_3}\end{matrix}@>>>1
\endCD
$$
where all vertical maps are inclusions of finite index. Hence we can choose $S_2\subset T_2$
such that $p_2(S_2)$ is an independent set of free generators of the free Abelain group
$(\pi_2\cdot\sqrt{\pi_3})/\sqrt{\pi_3}$ and $p_2(T_2-S_2)$ is an independent set of torsion generators of ${\pi_2}/{\pi_3}$. Note that $S_2\subset\pi_2\subset\sqrt{\pi_2}$.
Because the right-most vertical inclusion is finite index, we can choose $\calS_2\subset\sqrt{\pi_2}$
such that $q_2(\calS_2)$ is an independent set of free generators of
$\sqrt{\pi_2}/\sqrt{\pi_3}$, and for each $\sigma_2\in \calS_2$ there are unique $\ell_2\ge1$
and unique $\tau_{2*}\in S_2$
such that ${\sigma_2}^{\ell_2}=\tau_{2*}$ modulo $\sqrt{\pi_3}$.  Remark also that $\#S_2=\#\calS_2$.

Continuing in this way, we obtain $\left\{S_1,\cdots,S_c\right\}\subset\left\{T_1,\cdots,T_c\right\}$ and $\left\{\calS_1,\cdots,\calS_c\right\}$ such that
\begin{itemize}
\item $S_j\subset T_j$, $\#S_j=\#\calS_j$,
\item $p_j(S_j)$ is an independent set of free generators of ${\pi_j}/{\pi_{j+1}}$,
\item $p_j(T_j-S_j)$ is an independent set of torsion generators of ${\pi_j}/{\pi_{j+1}}$,
\item $q_j(\calS_j)$ is an independent set of free generators of
$\sqrt{\pi_j}/\sqrt{\pi_{j+1}}$,
\item for each $\sigma_j\in \calS_j\subset\sqrt{\pi_j}$, there exist unique $\ell_j\ge1$ and $\tau_{j*}\in S_j$
such that ${\sigma_j}^{\ell_j}=\tau_{j*}\mod\sqrt{\pi_{j+1}}$.
\end{itemize}

The adapted central series of $\pi$ allows us to choose a {\bf preferred basis} $\bfa$ of $\pi$;
we can choose $\bfa$ to be $\left\{\calS_1,\cdots,\calS_c\right\}$ so that it generates $\pi$
and $\pi$ can be embedded as a lattice of a connected simply connected nilpotent Lie group $G$,
the Malcev completion of $\pi$.
Its Lie algebra $\frakG$ has a linear basis $\log\bfa=\left\{\log \calS_1,\cdots,\log \calS_c\right\}$.
From ${\sigma_j}^{\ell_j}=\tau_{j*}\mod\sqrt{\pi_{j+1}}$, we have
\begin{align}\label{B}
\ell_j\log(\sigma_j)=\log({\sigma_j}^{\ell_j})=\log(\tau_{j*}) \mod \gamma_{j+1}(\frakG).\tag{B}
\end{align}
This induces that $\left\{\log S_1,\cdots,\log S_c\right\}$ is also a linear basis of $\frakG$.

Let $\phi:\pi\to\pi$ be an endomorphism. Then $\phi$ induces endomorphisms
$$
\phi_j:\pi_j\to\pi_j,\
\hat\phi_j:\pi/\pi_j\to\pi/\pi_j,\
\bar\phi_j:\pi_j/\pi_{j+1}\to\pi_j/\pi_{j+1},
$$
and
$$
\varphi_j:\sqrt{\pi_j}\to\sqrt{\pi_j},\
\hat{\varphi}_j:\pi/\sqrt{\pi_j}\to\pi/\sqrt{\pi_j},\
\bar{\varphi}_j:\sqrt{\pi_j}/\sqrt{\pi_{j+1}}\to\sqrt{\pi_j}/\sqrt{\pi_{j+1}}.
$$
Moreover, any endomorphism $\phi$ on $\pi$ extends uniquely
to a Lie group endomorphism $D$ on $G$, called the Malcev completion of $\phi$.
With respect to the preferred basis $\log\bfa$ of the Lie algebra $\frakG$ of $G$,
we can express the linearization $D_*$ of $\phi$ as a lower triangular block matrix,
each diagonal block $D_j$ is an integer matrix representing the endomorphism
$\bar\varphi_j:\sqrt{\pi_j}/\sqrt{\pi_{j+1}}\cong\bbz^{k_j}\to\sqrt{\pi_j}/\sqrt{\pi_{j+1}}\cong\bbz^{k_j}$.
For details, we refer to for example \cite{KL05}.
When the new basis $\left\{\log S_1,\cdots,\log S_c\right\}$ is used
instead of $\log\bfa$, the integer entries of block matrices $D_j$ will be changed to rational entries
because of the identities \eqref{B}, but the eigenvalues of $D_j$ will be unchanged.
This means that whenever the eigenvalues of $D_*$ are concerned, we may assume that $\pi_j/\pi_{j+1}$ is torsion-free, or $\pi_j=\sqrt{\pi_j}$,
or we may take tensor with $\bbr$.
Thus,
$$
\GR(\bar\phi_j)=\sp(D_j)=\GR(\bar\varphi_j).
$$
Now the theorem follows from $\GR(\phi)= \max_{j\ge1}\left\{\GR(\bar\phi_j)^{1/j}\right\}$.
\end{proof}

\begin{Example}\label{Ex1}
Let $\Nil$ be the $3$-dimensional Heisenberg group. That is,
$$
\Nil=\left\{\left[\begin{matrix}%
1&x&z\\0&1&y\\0&0&1\end{matrix}\right] \Big| \,
x,y,z\in\bbr\right\}.
$$
Consider the subgroups $\Gamma_k$, $k\in\bbn$, of $\Nil$,
$$
\Gamma_k=\left\{\left[\begin{matrix}1&n&\frac{\ell}{k}\\
0&1&m\\0&0&1\end{matrix}\right]\Big|\,\, m,n,\ell\in\bbz\right\}.
$$
These are lattices of $\Nil$, and every lattice of $\Nil$ is isomorphic to some $\Gamma_k$. Let
$$
a_1=\left[\begin{matrix}1&0&0\\0&1&1\\0&0&1\end{matrix}\right],\
a_2=\left[\begin{matrix}1&1&0\\0&1&0\\0&0&1\end{matrix}\right],\
a_3=\left[\begin{matrix}1&0&\frac{1}{k}\\0&1&0\\0&0&1\end{matrix}\right].
$$
Then $S=\left\{a_1,a_2,a_3\right\}$ is a generating set of $\Gamma_k$ satisfying $[a_1,a_2]=a_3^{-k}, [a_1,a_3]=[a_2,a_3]=1$,
and in fact
$$
\left[\begin{matrix}1&n&\frac{\ell}{k}\\0&1&m\\0&0&1\end{matrix}\right]=a_1^m a_2^na_3^\ell.
$$

Let $\pi=\Gamma_k=\langle a_1,a_2,a_3\mid [a_1,a_2]=a_3^{-k},[a_1,a_3]=[a_2,a_3]=1\rangle$.
Let $\pi'=\langle a_3\rangle$ and $S'=\left\{a_3\right\}$. Since $(a_3^{-k})^{n^2}=[a_1^n,a_2^n]$,
we have $L((a_3^{-k})^{n^2},S')=kn^2$ and $L((a_3^{-k})^{n^2},S)=4n$. Hence $$
L((a_3^{-k})^{n^2},S')>L((a_3^{-k})^{n^2},S)
$$
for all $n$ with $n>4/k$. It follows that $\pi'$ is distorted.

Consider any endomorphism $\phi:\pi\to\pi$. Then $\phi$ must be of the form
$$
\phi(a_1)=a_1^{m_{11}}a_2^{m_{21}}a_3^{p},\
\phi(a_2)=a_1^{m_{12}}a_2^{m_{22}}a_3^{q},\
\phi(a_3)=a_3^{m_{11}m_{22}-m_{12}m_{21}}.
$$
We will compute $\GR(\phi)$. The lower central series of $\pi$ is $\pi=\pi_1\supset\pi_2=\langle a_3^k\rangle$,
and its adapted central series is $\pi=\pi_1\supset\sqrt{\pi_2}=\langle a_3\rangle$.
We observe that $T_1=\left\{a_1,a_2,a_3\right\}$ and $T_2=\left\{a_3^k\right\}$ are sets satisfying the conditions of Lemma~\ref{W-basis}.
Then we can see that $S_1=\left\{a_1,a_2\right\}\subset T_1$ and $S_2=\left\{a_3^k\right\}\subset T_2$, and $S_1'=\left\{a_1,a_2\right\}$ and $S_2'=\left\{a_3\right\}$.
Further, $\left\{S_1',S_2'\right\}=\left\{a_1,a_2,a_3\right\}$ is a preferred basis for $\pi$.
The linearization of $\phi$ with respect to this preferred basis has two integer blocks $D_1$ and $D_2$,
where
$$
D_1=\left[\begin{matrix}m_{11}&m_{12}\\m_{21}&m_{22}\end{matrix}\right],\quad
D_2=\left[\begin{matrix}m_{11}m_{22}-m_{12}m_{21}\end{matrix}\right]
=\left[\begin{matrix}\det(D_1)\end{matrix}\right].
$$
By Theorem~\ref{GR-nilp}, we have $\GR(\phi)=\max\left\{\sp(D_1),\sp(D_2)^{1/2}\right\}$.
Let $\mu,\nu$ be the eigenvalues of $D_1$. Then
$$
\GR(\phi)=\max\left\{|\mu|,|\nu|,\sqrt{|\mu\nu|}\right\}=\max\left\{|\mu|,|\nu|\right\}=\sp(D_1).
$$
In fact, we will show in Theorem~\ref{Ko type} that it is always the case that $\GR(\phi)=\sp(D_1)$.
\end{Example}

We consider another example in which we obtain much information about linearizations of endomorphisms
and then we obtain an idea of proving the next result, Theorem~\ref{Ko type}.

\begin{Example}\label{Ex3}
Consider a {$2$-step} torsion-free nilpotent group $\pi$ generated by
$$
\tau_1,\tau_2,\tau_3,\sigma_{12},\sigma_{13}
$$
satisfying the relations
\begin{align*}
&[\tau_1,\tau_2]=\sigma_{12},[\tau_1,\tau_3]=\sigma_{13},
[\tau_2,\tau_3]=\sigma_{12}^m\sigma_{13}^{n},
[\tau_i,\sigma_{jk}]=[\sigma_{12},\sigma_{13}]=1.
\end{align*}
Since $\pi_2=\langle \sigma_{12},\sigma_{13}\rangle=\bbz^2$
and $\pi/\pi_2=\langle \bar\tau_1,\bar\tau_2,\bar\tau_3\rangle=\bbz^3$,
it follows that the set $\left\{T_1,T_2\right\}=\left\{\tau_1,\tau_2,\tau_3,\sigma_{12},\sigma_{13}\right\}$
satisfies the conditions of Lemma~\ref{W-basis}
and forms a preferred basis of our group $\pi$.
Let $\phi$ be an endomorphism of $\pi$. A direct computation shows that if
\begin{align*}
&\phi(\tau_i)=\tau_1^{d_{1i}}\tau_2^{d_{2i}}\tau_3^{d_{3i}}\mod\pi_2,
\end{align*}
i.e., if the first block of the linearization of $\phi$ is
$$
D_1=\left[\begin{matrix}
d_{11}&d_{12}&d_{13}\\
d_{21}&d_{22}&d_{23}\\
d_{31}&d_{32}&d_{33}\end{matrix}\right],
$$
then with $\sigma_{23}=[\tau_2,\tau_3]$, we have
\begin{align*}
&\phi(\sigma_{12})
=\sigma_{12}^{M_{33}}\sigma_{13}^{M_{23}}\sigma_{23}^{M_{13}},\
\phi(\sigma_{13})
=\sigma_{12}^{M_{32}}\sigma_{13}^{M_{22}}\sigma_{23}^{M_{12}},\
\phi(\sigma_{23})
=\sigma_{12}^{M_{31}}\sigma_{13}^{M_{21}}\sigma_{23}^{M_{11}},
\end{align*}
where $M_{ij}$ denote the $(i,j)$-minor of $D$.
These yield a matrix
$$
K=\left[\begin{matrix}M_{33}&M_{32}&M_{31}\\
M_{23}&M_{22}&M_{21}\\M_{13}&M_{12}&M_{11}
\end{matrix}\right]={\bigwedge}^2 \left(D_1\right),
$$
the second exterior power of $D_1$.
On the other hand, since $\sigma_{23}=\sigma_{12}^m\sigma_{13}^n$, we have
\begin{align}
\label{1}
&\phi(\sigma_{12})=\sigma_{12}^{M_{33}}\sigma_{13}^{M_{23}}\sigma_{23}^{M_{13}}
=\sigma_{12}^{M_{33}+mM_{13}}\sigma_{13}^{M_{23}+nM_{13}},\\
\label{2}
&\phi(\sigma_{13})=\sigma_{12}^{M_{32}}\sigma_{13}^{M_{22}}\sigma_{23}^{M_{12}}
=\sigma_{12}^{M_{32}+mM_{12}}\sigma_{13}^{M_{22}+nM_{12}},\\
\label{3}
&\phi(\sigma_{23})=\phi(\sigma_{12})^m\phi(\sigma_{13})^n
=\sigma_{12}^{M_{31}}\sigma_{13}^{M_{21}}\sigma_{23}^{M_{11}}
=\sigma_{12}^{M_{31}+mM_{11}}\sigma_{13}^{M_{21}+nM_{11}}.
\end{align}
From \eqref{1} and \eqref{2}, the second block of the linearization of $\phi$ is
$$
D_2=\left[\begin{matrix}M_{33}+mM_{13}&M_{32}+mM_{12}\\
M_{23}+nM_{13}&M_{22}+nM_{12}
\end{matrix}\right].
$$
Plugging \eqref{1} and \eqref{2} into \eqref{3}, we have
\begin{align}
\label{4}
\begin{cases}\left[\begin{matrix}M_{31}\\M_{21}\\M_{11}\end{matrix}\right]
=m\left[\begin{matrix}M_{33}\\M_{23}\\M_{13}\end{matrix}\right]
+n\left[\begin{matrix}M_{32}\\M_{22}\\M_{12}\end{matrix}\right]&
\text{when $(m,n)\ne(0,0)$,}\\
\left[\begin{matrix}M_{31}\\M_{21}\end{matrix}\right]
=\left[\begin{matrix}0\\0\end{matrix}\right]&\text{when $m=n=0$.}
\end{cases}
\end{align}
When $(m,n)\ne(0,0)$, because of \eqref{4}, $K$ is column equivalent to the matrix $K'$ with  the zero third column,
and then by doing some row operations on $K'$ we can see that $K'$ is row equivalent to the matrix $K''$
where
$$
K\rightsquigarrow K'=\left[\begin{matrix}M_{33}&M_{32}&0\\
M_{23}&M_{22}&0\\M_{13}&M_{12}&0
\end{matrix}\right]
\rightsquigarrow
K''=\left[\begin{BMAT}{ccc.c}{ccc.c}
M_{33}+mM_{13}&M_{32}+mM_{12}&&0\\
M_{23}+nM_{13}&M_{22}+nM_{12}&&0\\&&&\\
M_{13}&M_{12}&&0
\end{BMAT}\right].
$$
Thus the second block $D_2$ of the linearization $D_*$ is a block submatrix of $K''$.
This is obtained by removing the row and column of $K''$
that are determined by \eqref{3}
or by the relation $[\tau_2,\tau_3]=\sigma_{12}^m\sigma_{13}^{n}$.
Note also that $K$, $K'$ and $K''$ have the same eigenvalues which are
$0$ and the eigenvalues of $D_2$.
When $(m,n)=(0,0)$, because of \eqref{4}, we have
$$
K=\left[\begin{BMAT}{ccc.c}{ccc.c}
M_{33}&M_{32}&&0\\
M_{23}&M_{22}&&0\\&&&\\M_{13}&M_{12}&&M_{11}
\end{BMAT}\right].
$$
Thus $D_2$ of $D_*$ is a block submatrix of $K$, and
$K$ has $M_{11}$ and the eigenvalues of $D_2$ as its eigenvalues.

On the other hand, if $\mu_1,\mu_2,\mu_3$ are the eigenvalues of $D_1$, then as $K=\bigwedge^2(D_1)$, the eigenvalues of $K$ are $\mu_i\mu_j$ $(i<j)$. Consequently, we have
$$
\sp(D_1)=\max_{i=1,2,3}\left\{|\mu_i|\right\}\ge\max_{i\ne j}\left\{\sqrt{|\mu_i\mu_j|}\right\}=\sp(K)\ge\sp(D_2)^{1/2}.
$$
This proves that $\GR(\phi)=\sp(D_1)$.
\end{Example}

The following result was proved in \cite{Koberda} when $\phi$ is an automorphism
using the intrinsic polynomial structure of nilpotent groups.
We will now improve {\cite[Theorem~1.2]{Koberda}} from automorphisms to endomorphisms by using completely different arguments.


\begin{Thm}\label{Ko type}
Let $\phi:\pi\to\pi$ be an endomorphism on a finitely generated torsion-free nilpotent group $\pi$
with Malcev completion $D$. 
Then
$$
\GR(\phi)= \GR(\phi_{\mathrm{ab}}),
$$
where $\phi_{\mathrm{ab}}:\pi/[\pi,\pi]\to\pi/[\pi,\pi]$ be the endomorphism induced by $\phi$.
{Hence $\GR(\phi)=\sp(D_1)\le\sp(D_*)$.}
\end{Thm}

\begin{proof}
Let $\pi$ be $c$-step and choose a family of finite sets $\left\{T_1,\cdots,T_c\right\}$ satisfying the conditions of Lemma~\ref{W-basis}.
As it was observed in the proof of Theorem~\ref{GR-nilp},
we can choose $\left\{S_1,\cdots,S_c\right\}$ such that each $S_j\subset T_j\subset\pi_j$ projects onto free generators of $\pi_j/\pi_{j+1}$ and a preferred basis $\left\{S_1', \cdots, S_c'\right\}$ of $\pi$ so that each block matrix $D_j$ of the linearization $D_*$ of $\phi$ which is determined by $\log S_j'$ may be assumed to be determined by $\log S_j$.

Indeed, for each $j$ with $1\le j\le c$, we write $S_j=\left\{\tau_{j1}, \cdots,\tau_{jk_j}\right\}\subset T_j$; then if $j>1$, every $\tau_{jr}$ is of the form $[\tau_{1i},\tau_{j-1,\ell}]$. For $1\le j\le c$, if
$$
\phi(\tau_{j\ell})=\tau_{j1}^{d^j_{1\ell}}\cdots\tau_{jk_j}^{d^j_{k_j\ell}}\ \text{ modulo $\pi_{j+1}$},
$$
then the $j$th block of the linearization $D_*$ of $\phi$ is
$$
D_j=\left[\begin{matrix}
d^j_{11}&\cdots&d^j_{1k_j}\\
\vdots&&\vdots\\
d^j_{k_11}&\cdots&d^j_{k_jk_j}
\end{matrix}\right].
$$

In order to compare first the eigenvalues of $D_1$ with those of $D_2$, {we use the following new notation:}  $D_1=[d_{ij}^1]=[d_{ij}]$, $\sigma_{ij}=[\tau_{1i},\tau_{1j}]$
for all $1\le i<j\le k_1$. Then $\sigma_{ij}=\tau_{2,\ell}^{\pm1}\in S_2$
for some $\ell$ or $\sigma_{ij}$ is an word of elements in $S_2^{\pm1}$ modulo $\pi_3$
(see the presentation of $\pi$ in Example~\ref{Ex3}).
Let $S=\left\{\sigma_{ij}\mid1\le i<j\le k_1\right\}$; then we may assume that $S_2\subset S$.
Further, $S_2$ differs from $S$ except possibly by $\sigma_{ij}$'s,
words of elements in $S_2^{\pm1}$ modulo $\pi_3$
(note in Example~\ref{Ex3} that $S_2=\left\{\sigma_{12},\sigma_{13}\right\}$
and $S=\left\{\sigma_{12},\sigma_{13},\sigma_{12}^m\sigma_{13}^n\right\}$).

Now we can express $\phi(\sigma_{ij})$ as follows:
\begin{align}\label{P}
\phi(\sigma_{ij})=\sigma_{12}^{M_{1,2}^{i,j}}\sigma_{13}^{M_{1,3}^{i,j}}
\cdots\sigma_{1k_1}^{M_{1,k_1}^{i,j}}\cdots\sigma_{k_1-1,k_1}^{M_{k_1-1,k_1}^{i,j}}\
\text{ modulo $\pi_3$}\tag{P}
\end{align}
{for some integers $M^{i,j}_{p,q}$. We denote by $K$ the $\binom{k_1}{2}\x\binom{k_1}{2}$ matrix {$\left[M^{i,j}_{p,q}\right]$}
$$
K=\left[\begin{matrix}
M_{1,2}^{1,2}&M_{1,2}^{1,3}&\cdots&M_{1,2}^{k_1-1,k_1}\\
M_{1,3}^{1,2}&M_{1,3}^{1,3}&\cdots&M_{1,3}^{k_1-1,k_1}\\
\vdots&\vdots&&\vdots\\
M_{k_1-1,k_1}^{1,2}&M_{k_1-1,k_1}^{1,3}&\cdots&M_{k_1-1,k_1}^{k_1-1,k_1}
\end{matrix}\right].
$$
We will refer to the column vector
$\left({M_{1,2}^{i,j}},{M_{1,3}^{i,j}},\cdots,{M_{1,k_1}^{i,j}},\cdots,{M_{k_1-1,k_1}^{i,j}}\right)^t$
of $K$ as the $(i,j)$-column of $K$.
Remark that:
\begin{enumerate}
\item[(i)] For any $\sigma_{ij}\in S$, $M_{p,q}^{i,j}$ is unique for which $\sigma_{pq}\in S_2$.
\item[(ii)] If $\sigma_{ij}\in S-S_2$, then $\sigma_{ij}$ is an word $w$ of elements in $S_2^{\pm1}$ modulo $\pi_3$. If $w\ne1$ modulo $\pi_3$, the $(i,j)$-column of $K$ is an integer combination of $(p,q)$-columns of $K$ corresponding to the elements $\sigma_{pq}$ appearing in the word $w$. If $w\equiv1$, then
$M_{p,q}^{i,j}=0$ for which $\sigma_{pq}\in S_2$.
\item[(iii)] The right-hand side of the expression \eqref{P} can be rewritten
in terms of only the elements of $S_2$ using the words $\sigma_{ij}\equiv w(\sigma_{pq})$.
This yields the second block $D_2$.
\end{enumerate}}

Since $\sigma_{ij}=[\tau_{1i},\tau_{1j}]$, taking $\phi$ on both sides,
we have (see \cite[Lemma~4.1]{FFK} or \cite[p.~\!93, Lemma~4.1]{Passman})
\begin{align}\label{R}
\sigma_{12}^{M_{1,2}^{i,j}}\sigma_{13}^{M_{1,3}^{i,j}}\cdots\sigma_{1k_1}^{M_{1,k_1}^{i,j}}
\cdots\sigma_{k_1-1,k_1}^{M_{k_1-1,k_1}^{i,j}}
&=[\tau_{11}^{d_{1,i}}\cdots\tau_{1k_1}^{d_{k_1,i}},\tau_{11}^{d_{1,j}}\cdots\tau_{1k_1}^{d_{k_1,j}}]\notag\\
&=\prod_p\prod_q[\tau_{1p}^{d_{p,i}},\tau_{1q}^{d_{q,j}}]
=\prod_p\prod_q[\tau_{1p},\tau_{1q}]^{d_{p,i}d_{q,j}}\tag{R}\\
&=\prod_{1\le p<q\le k_1}\sigma_{pq}^{d_{p,i}d_{q,j}-d_{p,j}d_{q,i}}\ \text{ modulo $\pi_3$}.\notag
\end{align}
This shows that $K$ is the second exterior power of the matrix $D_1$, i.e., $K=\bigwedge^2(D_1)$.
Hence, if $\mu_i$ ($1\le i\le k_1$) are the eigenvalues of $D_1$, then $\mu_i\mu_j\ (i< j)$
are the eigenvalues of $K$.

From part (ii) of the above remarks, we see that $K$ is column equivalent to the matrix $K'$ with zero $(i,j)$-column for which $\sigma_{ij}=w(\sigma_{pq})\ne1$ modulo $\pi_3$. We rearrange the elements of $S$ so that
$S=S_2\cup(S-S_2)=S_2\cup S_2^1\cup S_2^2$ where $S_2^1=\left\{\sigma_{ij}\in S-S_2\mid \sigma_{ij}=1\right\}$ and
$S_2^2=\left\{\sigma_{ij}\in S-S_2\mid \sigma_{ij}\ne1\right\}$.
By rearranging $S$ to $S_2\cup(S-S_2)$, we have
$$
K\ \sim_C\  K'=
\left[\begin{BMAT}{ccc.ccc}{ccc.ccc}
&& &&& \\
&K'& &&0& \\
&& &&& \\
&& &&& \\
&*& &&*& \\
&& &&&
\end{BMAT}\right].
$$
The effect of part (iii) on $K$ and hence on $K'$ is doing some row operations
by using the $(i,j)$-rows in the last block of $K'$ for which $\sigma_{ij}=w(\sigma_{pq})\ne1$ modulo $\pi_3$.
By rearranging $S$ further to $S_2\cup S_2^1\cup S_2^2$, we have
$$
K'\ \sim_R\ K''=\left[\begin{BMAT}{ccc.ccc.c}{ccc.ccc.c}
&& &&& &\\
&D_2& &&0& &0\\
&& &&& &\\
&& &&& &\\
&*& &&*& &0\\
&& &&& &\\
&*& &&*& &0
\end{BMAT}\right].
$$
The middle block column is determined by that fact that if $\sigma_{ij}\equiv w\equiv1$, then
$M_{p,q}^{i,j}=0$ for which $\sigma_{pq}\in S_2$.

Consequently, the second block $D_2$ of $D_*$ is a block submatrix of $K''$ which is obtained by removing
the rows and columns associated to $S-S_2$.
Remark also that $K, K'$ and $K''$ have the same eigenvalues which contain the eigenvalues
of $D_2$.
This observation shows that
$$
\sp(D_1)=\max\left\{|\mu_i|\right\}\ge \max\left\{\sqrt{\mu_i\mu_j}\right\}=\sp(K)^{1/2}\ge\sp(D_2)^{1/2}.
$$

For the next inductive step, we recall that every element of $S_3(\subset T_3)$ is of the form $[\tau_{1\ell},\sigma_{ij}]$,
where $i<j$. Taking $\phi$, we have that
\begin{align*}
\phi([\tau_{1\ell},\sigma_{ij}])&=[\prod_r\tau_{1r}^{d_{r,\ell}},
\prod_{1\le p<q\le k_1}\sigma_{pq}^{d_{p,i}d_{q,j}-d_{p,j}d_{q,i}}]\\
&=\prod_r\prod_{1\le p<q\le k_1}[\tau_{1r},\sigma_{pq}]^{{d_{r,\ell}}(d_{p,i}d_{q,j}-d_{p,j}d_{q,i})}
\ \text{ modulo $\pi_4$}.
\end{align*}
This expression is unique except possibly the exponents of the elements $[\tau_{1r},\sigma_{pq}]=1$ modulo $\pi_4$.
This produces the matrix $K=D_1\bigotimes \bigwedge^2D_1$. First if $[\tau_{1r},\sigma_{pq}] =w(S_3)\ne1$ modulo $\pi_4$, by doing some column operations and then by doing some row operations we obtain a matrix $K''$, which can be regarded as a lower triangular block matrix. Finally, we remove the columns and rows from $K''$ which are associated with the elements $[\tau_{1r},\sigma_{pq}] =w(S_3)$ modulo $\pi_4$. This gives rise to the third block $D_3$ of $D_*$.
Hence $\sp(D_3)\le \sp(D_1)^3$.
Continuing in this way, we may assume that the $j$th block $D_j$ of $D_*$ is obtained from $\left(\bigotimes_{j-2}D_1\right)\bigotimes\bigwedge^2D_1$ so that
$$
\sp(D_1)\ge \sp(D_j)^{1/j}.
$$
Consequently, $\GR(\phi)=\max\left\{\sp(D_j)^{1/j}\right\}=\sp(D_1)=\GR(\phi_{\ab})\le \sp(D_*)$.
\end{proof}

\section{Lattices of $\Sol$}\label{Sol-lattice}

A group is said to have {max} if every its subgroup is finitely generated.
It is known that a solvable group has max if and only if it is polycyclic.
A polycyclic group is virtually poly-$\bbz$.
Following the proof of \cite[Lemma~3.1]{LL-JGP} again and using \cite[Lemma~4.4]{Raghunathan},
we can show that a polycyclic group has a torsion-free, fully invariant, finite index subgroup.
Consequently, for the computation of growth rates of endomorphisms on polycyclic groups,
we may assume that polycyclic groups are torsion-free, see Theorem~\ref{f-index}.

The simplest non-nilpotent poly-$\bbz$ group is the Klein bottle group.
We can compute easily the growth rate of any endomorphism on the Klein bottle group.

\begin{Example}\label{KB}
Consider the Klein bottle group $\pi=\langle x,y\mid yxy^{-1}=x^{-1}\rangle$
and let $\phi$ be any endomorphism on $\pi$.
Since $[\pi,\pi]=\langle x^2\rangle$, we have the induced endomorphisms $\phi'$ on $[\pi,\pi]$
and $\phi_\ab$ on $\pi_\ab=\langle \bar{x},\bar{y}\mid \bar{x}^2=[\bar{x},\bar{y}]=1\rangle$.
Recall for example from \cite[Lemma~2.1]{KLY} that $\phi$ satisfies $\phi(x)=x^q$ and $\phi(y)=y^rx^\ell$
for some integers $r,\ell$ and $q$ where either $r$ is odd, or $r$ is even and $q=0$.
Observe that $\phi'$ is the multiplication by $q$ on $[\pi,\pi]$ and so $\GR(\phi')=|q|$.
Since $\phi(y^2)=(y^rx^\ell)^2=y^{2r}x^{((-1)^r+1)\ell}$, it follows that $\phi_\ab$ is the multiplication by $r$
on the subgroup $\langle \bar{y}^2\rangle$ of $\pi_\ab$ and so $\GR(\phi_\ab)=|r|$.
Hence $\GR(\phi)\le\max\left\{|q|,|r|\right\}$. Now, to compute $\GR(\phi)$
we simply notice that the subgroups $\langle x^k\rangle$ of $\pi$ are \emph{undistorted}.
Consequently, we see that
$$
\GR(\phi)=\max\left\{|q|,|r|\right\}=\max\left\{\GR(\phi'),\GR(\phi_\ab)\right\}.
$$

Or, we note that $\pi=\langle x\rangle\rtimes\langle y\rangle$ is preserved by $\phi$,
inducing endomorphisms $\phi'$ and $\hat\phi$ so that $\phi'$ is the multiplication by $q$
and $\hat\phi$ is the multiplication by $r$ both on $\bbz$.
Since $\langle x\rangle$ is undistorted, we can conclude that
$\GR(\phi)=\max\left\{\GR(\phi'),\GR(\phi_\ab)\right\}$ as above.

{On the other hand, we remark also that the Klein group $\pi$ is a Bieberbach group.
Namely, the subgroup $\Gamma$ of $\pi$ generated by $x$ and $y^2$ is isomorphic to $\bbz^2$
which is of index $2$. By observation above, $\Gamma$ is $\phi$-invariant
and hence $\GR(\phi)=\GR(\phi|_\Gamma)=\sp(\phi|_\Gamma)$. Since $\phi|_\Gamma$ is represented by
a matrix with eigenvalues $q$ and $r$, we have $\GR(\phi)=\max\left\{|q|,|r|\right\}$ as it was observed above.}
\end{Example}

\medskip

The primary goal of this section is to compute the growth rate
of any endomorphism on a lattice of the $3$-dimensional solvable Lie group $\Sol$.
{To the best of our knowledge, such a computation on a poly-$\bbz$ group
was carried out for the first time in \cite[Theorem~5.1]{FFK}.
However, we will observe that its statement is false (Remark~\ref{counterexample}).}

The Lie group $\Sol$ is defined by $\Sol:=\bbr^2\rtimes_\varphi\bbr$ where
$$
\varphi(t)=\left[\begin{matrix}e^t&0\\0&e^{-t}\end{matrix}\right].
$$
Then $\Sol$ is a connected and simply connected unimodular $2$-step solvable Lie group of type $\R$.
It has a faithful representation into $\aff(\bbr^3)$ as follows:
$$
\Sol=\left\{\left[\begin{matrix}e^t&0&0&x\\0&e^{-t}&0&y\\0&0&1&t\\0&0&0&1\end{matrix}\right]\Big|\ x,y,t\in\bbr\right\}.
$$
The lattices $\Gamma$ of $\Sol$ are determined by $2\x2$-integer matrices $A$
$$
A=\left[\begin{matrix}\ell_{11}&\ell_{12}\\\ell_{21}&\ell_{22}\end{matrix}\right]
$$
of determinant $1$ and trace $>2$, see for example \cite[Lemma~2.1]{LZ-agt} or \cite{HL13}. Namely,
\begin{align}\label{S}
\Gamma=\GammaA=\langle a_1,a_2,\tau \mid[a_1,a_2]=1, \tau a_i\tau^{-1}=a_1^{\ell_{1i}}a_2^{\ell_{2i}}\rangle=\bbz^2\rtimes_A\bbz.\tag{S}
\end{align}

Let $\theta:\Gamma\to\Gamma$ be any endomorphism. By \cite[Theorem~2.4]{LZ-agt}, $\theta$ is one of the following:

\noindent
{\bf Type (I)} \hspace{.82cm}${\displaystyle \theta(a_1)=a_1^{m_{11}}a_2^{m_{21}},\
\theta(a_2)=a_1^{m_{21}}a_2^{m_{22}},\ \theta(\tau)=a_1^pa_2^q\tau}$

\noindent
{\bf Type (II)} \hspace{.66cm}${\displaystyle \theta(a_1)=a_1^{m_{11}}a_2^{m_{21}},\
\theta(a_2)=a_1^{m_{21}}a_2^{m_{22}},\ \theta(\tau)=a_1^pa_2^q\tau^{-1}}$

\noindent
{\bf Type (III)} \hspace{.5cm}${\displaystyle \theta(a_1)=1,\ \theta(a_2)=1,\
\theta(\tau)=a_1^pa_2^q\tau^m}$ \text{ with $m\ne\pm1$}

\noindent
Noting that $[a_1,a_2]=1$, we shall denote $\bfa^\bfx$ for $a_1^{x_1}a_2^{x_2}$.
Then $\tau a_i\tau^{-1}=\bfa^{A\bfe_i}$ and $\tau\bfa^\bfx\tau^{-1}=\bfa^{A\bfx}$.
Since $\Sol$ is of type $\R$, $\Sol$ satisfies the rigidity of lattices, see for example \cite{LL-nagoya}:
every endomorphism $\theta:\Gamma\to\Gamma$ extends uniquely as a Lie group endomorphism $D:\Sol\to\Sol$.
Let $S=\left\{a_1,a_2,\tau\right\}$, a set of generators for $\Gamma$.
With respect to the linear basis $\log{S}=\left\{\log a_1,\log a_2,\log\tau\right\}$ of the Lie algebra of $\Sol$,
the differential $D_*$ of $D$ can be expressed as a matrix of the form
$$
\left[\begin{matrix}m_{11}&m_{12}&0\\m_{21}&m_{22}&0\\{*}&*&\pm1\end{matrix}\right]\quad\text{or}\quad
\left[\begin{matrix}0&0&0\\0&0&0\\{*}&*&m\end{matrix}\right]
$$
according as $\theta$ is of type (I), (II) or (III). Let $H=\langle a_1,a_2\rangle=\bbz^2$,
the subgroup of $\Gamma$ generated by $a_1$ and $a_2$. Then $H$ is a fully invariant subgroup of $\Gamma$.

For any matrix $M$, $\bb M\bb_i$ denotes the sum of the absolute values of the entries of the $i$th column of $M$,
and the subscript is suppressed when $M$ is a column matrix. Let $S=\left\{a_1,a_2,\tau\right\}$ and $S'=\left\{a_1,a_2\right\}$.
Since $\tau^n\bfa^\bfx\tau^{-n}=\bfa^{A^n\bfx}$ for all $n>0$, we have
\begin{align*}
&L(\bfa^{A^n\bfx},S)=L(\tau^n\bfa^\bfx\tau^{-n},S)=2n+\bb\bfx\bb,\\
&L(\bfa^{A^n\bfx},S')=\bb A^n\bfx\bb_1+\bb A^n\bfx\bb_2.
\end{align*}
Now, we recall from \cite[Sect.~3]{HL13} or \cite[Theorem~2.4]{LZ-agt} that $A$ is conjugate to a diagonal matrix
$$
\bba=\left[\begin{matrix}e^{t_0}&0\\0&e^{-t_0}\end{matrix}\right]
:=\left[\begin{matrix}\alpha&0\\0&\beta\end{matrix}\right],
$$
say $P^{-1}AP=\bba$ for some invertible matrix $P$. Hence $A^n=P\bba^n P^{-1}$ and this shows that
$$
A^n=\left[\begin{matrix}a_{11}\alpha^n+b_{11}\beta^n&a_{12}\alpha^n+b_{12}\beta^n\\
a_{21}\alpha^n+b_{21}\beta^n&a_{22}\alpha^n+b_{22}\beta^n\end{matrix}\right]
$$
for some nonzero constants $a_{ij}$ which are independent of $n$.
Thus $\bb A^n\bb_i=a_i|\alpha|^n+b_i|\beta|^n$ for some positive constants $a_i$ and $b_i$.
It now follows that
$$
L(\bfa^{A^n\bfx},S')=A_1|\alpha|^n+A_2|\beta|^n
$$
for some positive constants $A_1$ and $A_2$ where $\alpha$ and $\beta$ are the eigenvalues of $A$.
Consequently $H$ is exponentially distorted.

Given $\theta$, we denote $\theta':H\to H$ and $\hat\theta:\Gamma/H\to\Gamma/H$. Then
$$
\GR(\theta)\le\max\left\{\GR(\theta'),\GR(\hat\theta)\right\}=\max\left\{\sp(\theta'),\sp(\hat\theta)\right\}=\sp(D_*).
$$
We now determine $\GR(\theta)$ explicitly:
\smallskip

\noindent
\underline{When $\theta$ is of type (III)}, we have $\theta'$ is a trivial homomorphism
and $\hat\theta$ is multiplication by $m$. Because $\GR(\theta')=0$ and $\GR(\hat\theta)=|m|$,
we have
$$
|m|=\GR(\hat\theta)\le\GR(\theta)\le \max\left\{\GR(\theta'),\GR(\hat\theta)\right\}=|m|.
$$
\smallskip

\noindent
\underline{When $\theta$ is of type (I)}, $\theta'$ is represented by the matrix
$$
M=\left[\begin{matrix}m_{11}&m_{12}\\m_{21}&m_{22}\end{matrix}\right].
$$
That is, $\theta(a_i)=\bfa^{M\bfe_i}$. From this, we can show that $\theta(\bfa^\bfp\tau)=\bfa^{M\bfp+\bfp}\tau$
where $\bfp$ is any column vector $\bfp=(p,q)^t$.
Furthermore, we have
\begin{align*}
&\theta^k(a_i)=\bfa^{M^k\bfe_i}=\bfa^{A^{n}A^{-n}M^k\bfe_i}=\tau^{n}\bfa^{A^{-n}M^k\bfe_i}\tau^{-n},\\
&\theta^k(\tau)=\bfa^{(I+M+M^2+\cdots+M^{k-1})\bfp}\tau\\
&\qquad\ =\bfa^\bfp(\tau^{n_1}\bfa^{A^{-n_1}M\bfp}\tau^{-n_1})
\cdots(\tau^{n_{k-1}}\bfa^{A^{-n_{k-1}}M^{k-1}\bfp}\tau^{-n_{k-1}})\tau.
\end{align*}
These identities imply that
\begin{align*}
&L(\theta^k(a_i),S)=\min_{n\ge0}\left\{\bb A^{-n}M^k\bb_i+2n\right\},\\
&L(\theta^k(\tau),S)=\min_{n_i\ge0}\left\{\bb\bfp\bb+\bb A^{-n_1}M\bfp\bb
+\cdots+\bb A^{-n_{k-1}}M^{k-1}\bfp\bb\right.\\
&\hspace{3.3cm}+2\left(n_1+\cdots+n_{k-1}\right)+1\Big\}\ \text{ when $\bfp\ne\bf0$}.
\end{align*}

Let $\mu$ and $\nu$ be the eigenvalues of $M$. If $\mu=0$ or $\nu=0$ then it follows that $M=0$.
Indeed, using the notation of  \cite[Theorem~2.4]{LZ-agt},
$$
\mu\text{ or }\nu =\tfrac{2(\ell_{11}\ell_{22}-1)u-\left(\ell_{11}\ell_{22}
-\ell_{12}\ell_{21}\pm\ell_{21}\sqrt{(\ell_{11}+\ell_{22})^2-4}\right)v}
{2(\ell_{11}\ell_{22}-1)},
$$
and $\mu$ or $\nu=0$ induces that $v=0$ (because $\sqrt{(\ell_{11}+\ell_{22})^2-4}$ is an irrational number)
and so $\mu=\nu=u=0$ and hence $M=0$.
This case yields $\GR(\theta)=\GR(\hat\theta)=1$. Hence we now {assume $M\ne0$},
or equivalently, $\mu\ne0$ and $\nu\ne0$. Recall from the proof of \cite[Theorem~2.4]{LZ-agt}
that both $M$ and $A$ are simultaneously diagonalizable to diagonal matrices $\diag\left\{\mu,\nu\right\}$
and $\diag\left\{\alpha,\beta\right\}$, respectively. {\bf In what follows, we may assume that $\alpha>1$
and $\beta=\frac{1}{\alpha}$}. Thus $A^{-n}M^k$ is similar to
the diagonal matrix $\bbd=\diag\left\{\frac{\mu^k}{\alpha^n},\frac{\nu^k}{\beta^n}\right\}$,
say $A^{-n}M^k=P\bbd P^{-1}$ for some invertible matrix $P$.
This shows that each entry of $A^{-n}M^k$ is a linear combination of $\frac{\mu^k}{\alpha^n}$
and $\frac{\nu^k}{\beta^n}$ by nonzero constants which are independent of $n$ and $k$
(see the discussion before for $A$ and $\bba$):
\begin{align*}
&A^{-n}M^k=\left[\begin{matrix}a_{11}\frac{\mu^k}{\alpha^n}+b_{11}\frac{\nu^k}{\beta^n}
&a_{12}\frac{\mu^k}{\alpha^n}+b_{12}\frac{\nu^k}{\beta^n}\\
a_{21}\frac{\mu^k}{\alpha^n}+b_{21}\frac{\nu^k}{\beta^n}
&a_{22}\frac{\mu^k}{\alpha^n}+b_{22}\frac{\nu^k}{\beta^n}
\end{matrix}\right],\\
&A^{-n_j}M^j\bfp=\left[\begin{matrix}
(p_1a_{11}+p_2a_{12})\frac{\mu^j}{\alpha^{n_j}}+(p_1b_{11}+p_2b_{12})\frac{\nu^j}{\beta^{n_j}}\\
(p_1a_{21}+p_2a_{22})\frac{\mu^j}{\alpha^{n_j}}+(p_1b_{21}+p_2b_{22})\frac{\nu^j}{\beta^{n_j}}
\end{matrix}\right]\end{align*}
This makes possible to assume that $\bb A^{-n}M^k\bb_i$ and $\bb A^{-n_j}M^j\bfp\bb$
are of the form $A_1\frac{|\mu|^k}{\alpha^{n}}+A_2\frac{|\nu|^k}{\beta^{n}}$
and $B_1\frac{|\mu|^j}{\alpha^{n_j}}+B_2\frac{|\nu|^j}{\beta^{n_j}}$ respectively.

According to Lemma~\ref{tech1}, we may assume that $\bb A^{-n}M^k\bb_i$ and $\bb A^{-n_j}M^j\bfp\bb$
are simply $\frac{|\mu|^k}{\alpha^{n}}+\frac{|\nu|^k}{\beta^{n}}$
and $\frac{|\mu|^j}{\alpha^{n_j}}+\frac{|\nu|^j}{\beta^{n_j}}$, respectively.
Let
\begin{align}
\label{min1}
L(\theta^k(a_i),S)=\bb A^{-n_k}M^k\bb_i+2n_k\tag{M1}
\end{align}
for some $n_k$. Then
\begin{align*}
\lim_{k\to\infty}L(\theta^k(a_i),S)^{1/k}&=\lim_{k\to\infty}\left(\frac{|\mu|^k}{\alpha^{n_k}}+\frac{|\nu|^k}{\beta^{n_k}}\right)^{1/k}\notag\\
&=\max\left\{\lim_{k\to\infty}\left(\frac{|\mu|^k}{\alpha^{n_k}}\right)^{1/k},\
\lim_{k\to\infty}\left(\frac{|\nu|^k}{\beta^{n_k}}\right)^{1/k}\right\}.
\end{align*}
Let
\begin{align}
\label{min2}
L(\theta^k(\tau),S)=&\ \bb\bfp\bb+\bb A^{-n_1}M\bfp\bb+\cdots+\bb A^{-n_{k-1}}M^{k-1}\bfp\bb\tag{M2}\\
&\ +2\left(n_1+\cdots+n_{k-1}\right)+1\notag
\end{align}
for some $n_1,\cdots,n_{k-1}$. Then $\bb A^{-n_j}M^j\bfp\bb=\min_{n\ge0}\left\{\bb A^{-n}M^j\bfp\bb\right\}$ and
\begin{align*}
\lim_{k\to\infty}L(\theta^k(\tau),S)^{1/k}
&=\lim_{k\to\infty}\left(\bb A^{-n_1}M\bfp\bb+\cdots+\bb A^{-n_{k-1}}M^{k-1}\bfp\bb\right)^{1/k}\notag\\
&=\lim_{k\to\infty}\left(\sum_{j=1}^{k-1}\left(\frac{|\mu|^{j}}{\alpha^{n_{j}}}+\frac{|\nu|^{j}}{\beta^{n_{j}}}\right)\right)^{1/k}.
\end{align*}
Consequently, when $M\ne0$ we have
\begin{align}\label{G}
\GR(\theta)&=\max\left\{\lim_{k\to\infty}\left(\frac{|\mu|^k}{\alpha^{n_k}}\right)^{1/k},\
\lim_{k\to\infty}\left(\frac{|\nu|^k}{\beta^{n_k}}\right)^{1/k},\right.\tag{G}\\
&\hspace{3cm}\left.
\lim_{k\to\infty}\left(\sum_{j=1}^{k-1}\left(\frac{|\mu|^{j}}{\alpha^{n_{j}}}+\frac{|\nu|^{j}}{\beta^{n_{j}}}\right)\right)^{1/k}\right\}.\notag
\end{align}
Here, we used Lemma~\ref{max-lim} for the identity.

\begin{Example}\label{sol-ex1}
When $\theta$ is of type (I) with $M=A$, the formula \eqref{G} is already enough to compute $\GR(\theta)$. In fact,
\begin{align*}
&\theta^k(a_i)=\bfa^{M^k\bfe_i}=\tau^ka_i^{A^{-k}M^k\bfe_i}\tau^{-k}=\tau^ka_i\tau^{-k},\\
&\theta^k(\tau)=\bfa^{(I+M+M^2+\cdots+M^{k-1})\bfp}\tau\\
&\hspace{0.94cm}=\bfa^\bfp(\tau\bfa^{A^{-1}M\bfp}\tau^{-1})\cdots(\tau^{k-1}\bfa^{A^{-(k-1)}M^{k-1}\bfp}\tau^{-(k-1)})\tau\\
&\hspace{0.94cm}=\bfa^\bfp(\tau\bfa^{\bfp}\tau^{-1})
\cdots(\tau^{k-1}\bfa^\bfp\tau^{-(k-1)})\tau.
\end{align*}
This shows that $n_j=j$ for all $j\ge1$ in \eqref{min1} and \eqref{min2}. Because $\mu=\alpha$ and $\nu=\beta$,
$$
\lim_{k\to\infty}\left(\frac{|\mu|^j}{\alpha^{n_j}}\right)^{1/k}=
\lim_{k\to\infty}\left(\frac{|\nu|^j}{\beta^{n_j}}\right)^{1/k}=1
$$
and hence $\GR(\theta)=1$ whereas $\GR(\theta')=\sp(A)=\alpha>1$.
This provides a counter-example to \cite[Theorem~5.1]{FFK}.
\end{Example}

In order to use the formula \eqref{G} in general, we first have to determine $n_j$ $(j=1,\cdots,k$)
in \eqref{min1} and \eqref{min2}. For this purpose, we observe that
\begin{align}
\frac{|\mu|^j}{\alpha^{n_j}}+\frac{|\nu|^j}{\beta^{n_j}}
\le\frac{|\mu|^j}{\alpha^{n}}+\frac{|\nu|^j}{\beta^{n}}\
&\Longrightarrow\ \left|\frac{\mu}{\nu}\right|^j\frac{1}{\alpha^{n_j}}+\alpha^{n_j}
\le \left|\frac{\mu}{\nu}\right|^j\frac{1}{\alpha^{n}}+\alpha^{n}\notag\\
\label{I}
&\Longrightarrow\ \left|\frac{\mu}{\nu}\right|^j\frac{\alpha^n-\alpha^{n_j}}{\alpha^{n_j+n}}
\le \alpha^{n}-\alpha^{n_j}.\tag{$\dagger$}
\end{align}

First we consider the case where $|\mu|\le|\nu|$. In this case, we notice that all $n_j=0$.
If some $n_j>0$, then $0\le n<n_j$ for some $n$ and then the above inequality \eqref{I} reduces to
$$
\left|\frac{\mu}{\nu}\right|^j\ge \alpha^{n_j+n}>\alpha^{2n_j-1}>1,
$$
which contradicts $|\mu|\le|\nu|$. By \eqref{G}, we have
\begin{align*}
\GR(\theta)=\max\left\{|\mu|,|\nu|,\lim_{k\to\infty}\left(\sum_{j=1}^{k-1}\left(|\mu|^j+|\nu|^j\right)\right)^{1/k}\right\}.
\end{align*}
Since $\lim_{k\to\infty}\left(\sum_{j=1}^{k-1}|\mu|^j\right)^{1/k}=|\mu|$ and $\lim_{k\to\infty}\left(\sum_{j=1}^{k-1}|\nu|^j\right)^{1/k}=|\nu|$, Lemma~\ref{tech1} implies that $\GR(\theta)=|\nu|$.

Next we consider the case where $|\nu|<|\mu|$. Since $|\frac{\mu}{\nu}|^j>1>\alpha^{-1}$, we have
$$
\left|\frac{\mu}{\nu}\right|^j\ge\alpha^{2n_j-1}
$$
even for $n_j=0$. On the other hand, for $n>n_j$, the above inequality \eqref{I} reduces to
$$
\left|\frac{\mu}{\nu}\right|^j\le\alpha^{n_j}\cdot\min_{n>n_j}\left\{\alpha^n\right\}=\alpha^{n_j}\cdot\alpha^{n_j+1}
=\alpha^{2n_j+1}.
$$
In all, we obtain that
\begin{align}\label{ineq}
\alpha^{2n_j-1}\le\left|\frac{\mu}{\nu}\right|^j\le\alpha^{2n_j+1},\quad 1\le j\le k.\tag{$\ast$}
\end{align}
Taking $\log$ on \eqref{ineq},
we obtain
\begin{align*}
\frac{\left(\frac{\log|\mu|-\log|\nu|}{\log\alpha}\right)j-1}{2}\le n_j\le
\frac{\left(\frac{\log|\mu|-\log|\nu|}{\log\alpha}\right)j+1}{2}.
\end{align*}
Write $\chi=\frac{\log|\mu|-\log|\nu|}{\log\alpha}$.
First we compute
\begin{align*}
&\frac{|\mu|}{\alpha^{{(k\chi+1)}/{2k}}}
\le\left(\frac{|\mu|^k}{\alpha^{n_k}}\right)^{1/k}
\le\frac{|\mu|}{\alpha^{{(k\chi-1)}/{2k}}},\quad
\lim_{k\to\infty}\left(\frac{|\mu|^k}{\alpha^{n_k}}\right)^{1/k}
=\frac{|\mu|}{\alpha^{{\chi}/{2}}}.
\end{align*}
Similarly, we obtain that
$$
\lim_{k\to\infty}\left(\frac{|\nu|^k}{\beta^{n_k}}\right)^{1/k}
=\frac{|\nu|}{\beta^{{\chi}/{2}}}=\alpha^{{\chi}/{2}}|\nu|.
$$
Next, we compute
\begin{align*}
\alpha^{-1/2}\sum_{j=1}^{k-1}\left(\frac{|\mu|}{\alpha^{\chi/2}}\right)^j
\le\sum_{j=1}^{k-1}\frac{|\mu|^j}{\alpha^{n_j}}\le \alpha^{1/2}\sum_{j=1}^{k-1}
\left(\frac{|\mu|}{\alpha^{\chi/2}}\right)^j.
\end{align*}
Thus a simple computation shows that
\begin{align*}
&\lim_{k\to\infty}\left(\sum_{j=1}^{k-1}\frac{|\mu|^j}{\alpha^{n_j}}\right)^{1/k}=\frac{|\mu|}{\alpha^{\chi/2}},\quad
\lim_{k\to\infty}\left(\sum_{j=1}^{k-1}\frac{|\nu|^j}{\beta^{n_j}}\right)^{1/k}=\frac{|\nu|}{\beta^{\chi/2}}=\alpha^{\chi/2}|\nu|.
\end{align*}
In all, we have that if $|\nu|<|\mu|$ then
$$
\GR(\theta)=\max\left\{\frac{|\mu|}{\alpha^{\chi/2}},\
\frac{|\nu|}{\beta^{\chi/2}}\right\}=\max\left\{|\mu|\sqrt{\left|\frac{\nu}{\mu}\right|},
|\nu|\sqrt{\left|\frac{\mu}{\nu}\right|}\right\}
=\sqrt{\left|{\mu}{\nu}\right|}=\sqrt{|\det(M)|}.
$$
Here the first identity follows from Lemma~\ref{tech1} and the second identity follows from
\begin{align*}
&\alpha^{\chi/2}=\left(e^{\chi/2}\right)^{\log\alpha}=e^{\left(\log{\left|\frac{\mu}{\nu}\right|}\right)/2}=\sqrt{\left|\frac{\mu}{\nu}\right|}.
\end{align*}

In summary for endomorphisms of type (I),
\begin{Thm}\label{Sol-I}
Let $\theta$ be an endomorphism of type $(\mathrm{I})$. Then
$$
\GR(\theta)=\begin{cases}
1&\text{when $\mu=\nu=0$}\\
|\nu|&\text{when $0<|\mu|\le|\nu|$}\\
\sqrt{|\mu\nu|}&\text{when $0<|\nu|<|\mu|$.}
\end{cases}
$$
In particular, $\GR(\theta)$ is an algebraic integer.
\end{Thm}

In the following examples,
we will carry on explicit computation of growth rate to confirm our Theorem~\ref{Sol-I}.

\begin{Example}\label{sol-ex2}
Let us consider an example:
$$
A=\left[\begin{matrix}
2&1\\1&1\end{matrix}\right],\
M=\left[\begin{matrix}
1&\hspace{8pt}2\\2&-1\end{matrix}\right],\ \theta(\tau)=\tau.
$$
Then:
\begin{itemize}
\item $\alpha=\frac{3+\sqrt{5}}{2}, \beta=\frac{3-\sqrt{5}}{2},\mu=\sqrt{5}$ and $\nu=-\sqrt{5}$;
hence $|\mu|\le|\nu|$.
\item $\tau a_1\tau^{-1}=a_1^{2}a_2$ and $\tau a_2\tau^{-1}=a_1a_2$.
\item $\theta^{2k}(a_1)=a_1^{5^k}(=\tau a_1^{5^k}a_2^{-5^k}\tau^{-1}=\tau^{-1}a_1^{2\cdot5^k}a_2^{5^k}\tau),\
\theta^{2k}(a_2)=a_2^{5^k}$, $\theta^k(\tau)=\tau$.
\end{itemize}
Hence $L_{2k}(\theta,S)=5^k$ and so $\GR(\theta)=\sqrt{5}=|\nu|$.
\end{Example}

\begin{Example}\label{sol-ex3}
Let us consider another example:
$$
A=\left[\begin{matrix}
1&1\\2&3\end{matrix}\right],\
M=\left[\begin{matrix}
0&1\\2&2\end{matrix}\right],\ \theta(\tau)=\tau.
$$
Then:
\begin{itemize}
\item $\alpha=2+\sqrt{3}, \beta=2-\sqrt{3},\mu=1+\sqrt{3}$ and $\nu=1-\sqrt{3}$; $|\nu|<|\mu|$.
\item $\tau a_1\tau^{-1}=a_1a_2^2$ and $\tau a_2\tau^{-1}=a_1a_2^3$.
\item By induction on $k$, we can show that
\begin{align*}
&\theta^{2k+1}(a_2)=\tau^{k+1}a_1^{2^k}\tau^{-(k+1)},\\
&\theta^{2k}(a_2)=\tau^{k}a_2^{2^k}\tau^{-k}(=\tau^{k-1}a_1^{2^k}a_2^{3\cdot2^k}\tau^{-(k-1)}=\tau^{k+1}a_1^{-2^{k}}a_2^{2^k}\tau^{-(k+1)}),\\
&\theta^{2k+1}(a_1)=\tau^{k}a_2^{2^{k+1}}\tau^{-k},\\
&\theta^{2k}(a_1)=\tau^{k}a_1^{2^k}\tau^{-k},\\
&\theta^k(\tau)=\tau.
\end{align*}
\end{itemize}
The observation shows that  $L_{2k}(\theta,S)=2^k+2k$ and so
\begin{align*}
\GR(\theta)&=\lim_{k\to\infty}L_k(\theta,S)^{1/k}=\lim_{k\to\infty}L_{2k}(\theta,S)^{1/2k}
=\lim_{k\to\infty}(2^k+2k)^{1/2k}=\sqrt{2}.
\end{align*}
On the other hand, because $\det(M)=-2$, Theorem~\ref{Sol-I} also says that $\GR(\theta)=\sqrt{2}$.
\end{Example}
\smallskip

\noindent
\underline{When $\theta$ is of type (II)}, using the notation as in type (I),
we have $\theta^2(a_i)=\bfa^{M^2\bfe_i}$ and $\theta^2(\tau)=\bfa^{(M-A)\bfp}\tau$.
This says that $\theta^2$ is of type (I) with $M^2A=AM^2$. Let $\mu$ and $\nu$ be the eigenvalues of $M$
so that $\frac{\mu^2}{\alpha}$ and $\frac{\nu^2}{\beta}$ are the eigenvalues of $A^{-1}M^2$.
By Theorem~\ref{Sol-I}, we have
$$
\GR(\theta^2)=\begin{cases}
1&\text{when $\mu=\nu=0$}\\
|\nu|^2&\text{when $0<|\mu|\le|\nu|$}\\
|\mu\nu|&\text{when $0<|\nu|<|\mu|$.}
\end{cases}
$$
Consequently,
$$
\GR(\theta)=\GR(\theta^2)^{1/2}=\begin{cases}
1&\text{when $\mu=\nu=0$}\\
|\nu|&\text{when $0<|\mu|\le|\nu|$}\\
\sqrt{|\mu\nu|}&\text{when $0<|\nu|<|\mu|$.}
\end{cases}
$$
\bigskip

Now we can summarize what we have found as follows:
\begin{Thm}\label{Sol}
Let $\theta$ be an endomorphism on the group $\GammaA$ given as in \eqref{S}.
Let $\theta'=M$ and $\hat\theta=m$ be the endomorphisms induced by $\theta$
on the subgroup $H=\bbz^2$ and the quotient group $\pi/H=\bbz$, respectively.
Let $\alpha,\beta$ and $\mu,\nu$ are the respective eigenvalues of $A$ and $M$
so that $\frac{\mu}{\alpha},\frac{\nu}{\beta}$ are the eigenvalues of $A^{-1}M$ and $\alpha>1$.
Then
$$
\GR(\theta)=\begin{cases}
|m|=\GR(\hat\theta)&\text{when $\mu=\nu=0$}\\
|\nu|=\GR(\theta')&\text{when $0<|\mu|\le|\nu|$}\\
\sqrt{|\mu\nu|}&\text{when $0<|\nu|<|\mu|$.}
\end{cases}
$$
{In particular, $\GR(\theta)$ is an algebraic integer and $\GR(\theta)\le\sp(D_*)$.}
\end{Thm}

\begin{Rmk}\label{counterexample}
In Theorem~\ref{Sol}, if $0<|\nu|<|\mu|$ then
$$
\GR(\theta)=\sqrt{|\mu\nu|}<\sqrt{|\mu|^2}=|\mu|=\max\left\{\GR(\theta')=|\mu|,\GR(\hat\theta)=1\right\}.
$$
Thus there are many endomorphisms (of type (I) or type (II)) for which the strict inequality
$\GR(\theta)<\max\left\{\GR(\theta'),\GR(\hat\theta)\right\}$ holds. This provides a counterexample to \cite[Theorem~5.1]{FFK}.
\end{Rmk}

\begin{Rmk}
{Let $f$ be a continuous map on a compact manifold $M$.
Then by \cite[Theorem~8.1.1]{KH} we have that $h_\top(f)\ge h_\alg(f)=\log\GR(\phi)$,
where $\phi$ is a homomorphism induced by $f$ of the group of covering transformations
on the universal cover of $M$.}

When $M$ is an infra-solvmanifold of type $\R$, it is known from \cite[Theorem~5.2, Remark~5.3]{FL} that
$$
h_\top(f)\ge\log\sp(\bigwedge D_*).
$$

{Let $M$ be an infra-nilmanifold or a closed $3$-manifold with $\Sol$-geometry.
By Theorem~\ref{Ko type} and Theorem~\ref{Sol}, we can see that $\sp(D_*)\ge\GR(\phi)=\GR(f)$.}
Consequently, we have that
$$
h_\top(f)\ge\log\sp(\bigwedge D_*)\ge\log\sp(D_*)\ge \log\GR(f)=h_\alg(f).
$$

{We wonder whether the inequality $\sp(D_*)\ge \GR(f)$ holds
whenever $M$ is an infra-solvmanifold (of type $\R$).}
{For the relation between the growth rate $\GR(\phi)$ and asymptotic Reidemeister number, see \cite{FL1}.}
\end{Rmk}

\begin{Rmk}
From Theorem~\ref{GR-nilp} and Theorem~\ref{Sol} it follows that
the growth rate of any endomorphism on a finitely generated torsion-free nilpotent group or a lattice of the $3$-dimensional solvable Lie group $\Sol$ is an algebraic number.
The question of determining groups for which the growth rate of a group endomorphism is an algebraic number was raised by R.~Bowen in \cite[p.\!~27]{B78}.
\end{Rmk}

\end{document}